\documentclass[11pt]{amsart}
\usepackage{amsmath, amscd}
\usepackage{amsfonts}
\usepackage{amssymb}
\usepackage{mathtools}
\usepackage{mathrsfs}
\usepackage{graphicx}
\usepackage{epstopdf}
\usepackage{alias cnt}
\usepackage{overpic}
\usepackage{enumitem,times}
\usepackage{epigraph}
\usepackage[bookmarks=true,pdfstartview=FitH, pdfborder={0 0 0}, colorlinks=true,citecolor=blue, linkcolor=blue]{hyperref}

\newtheorem{theorem}{Theorem}[section]
\newaliascnt{conj}{theorem}
\newaliascnt{cor}{theorem}
\newaliascnt{lemma}{theorem}
\newaliascnt{fact}{theorem}
\newaliascnt{claim}{theorem}
\newaliascnt{prop}{theorem}
\newaliascnt{definition}{theorem}
\newaliascnt{assump}{theorem}
\newaliascnt{question}{theorem}
\newaliascnt{notation}{theorem}
\newaliascnt{convention}{theorem}

\newtheorem{conj}[conj]{Conjecture}

\newtheorem{lemma}[lemma]{Lemma}

\newtheorem{prop}[prop]{Proposition}
\newtheorem{definition}[definition]{Definition}
\newtheorem{question}[question]{Question}

\aliascntresetthe{conj} \aliascntresetthe{cor}
\aliascntresetthe{lemma} \aliascntresetthe{fact}
\aliascntresetthe{claim} \aliascntresetthe{prop}
\aliascntresetthe{definition} \aliascntresetthe{question}
\aliascntresetthe{notation} \aliascntresetthe{convention}
\aliascntresetthe{assump}

\makeatletter
\let\oldtheequation\theequation
\renewcommand\tagform@[1]{\maketag@@@{\ignorespaces#1\unskip\@@italiccorr}}
\renewcommand\theequation{(\oldtheequation)}
\makeatother

\theoremstyle{remark}
\newaliascnt{rmk}{theorem}
\newtheorem{remark}[rmk]{Remark}
\aliascntresetthe{rmk}

\theoremstyle{remark}
\newaliascnt{exam}{theorem}
\newtheorem{exam}[exam]{Example}
\aliascntresetthe{exam}

\def\sek~{\S{}}

\DeclarePairedDelimiter\abs{\lvert}{\rvert}

\DeclareMathOperator{\Sk}{Sk}

\newcommand{\Rbb}{\mathbb{R}}
\newcommand{\Zbb}{\mathbb{Z}}

\newcommand{\kbf}{\mathbf{k}}

\newcommand{\rbf}{\mathbf{r}}
\newcommand{\xbf}{\mathbf{x}}

\newcommand{\s}{\vskip.1in}
\newcommand{\n}{\noindent}
\newcommand{\p}{\partial}

\newcommand{\be}{\begin{enumerate}}
\newcommand{\ee}{\end{enumerate}}
\newcommand{\op}{\operatorname}

\begin{document}

\title{A dynamical construction of Liouville domains}

\author{Yang Huang}
\address{University of Munich}
\email{hymath@gmail.com} \urladdr{https://sites.google.com/site/yhuangmath}



\begin{abstract}
We first present a general construction of Liouville domains as partial mapping tori. Then we study two examples where the (partial) monodromies exhibit certain hyperbolic behavior in the sense of Dynamical Systems. The first example is based on Smale's attractor, a.k.a., solenoid; and the second example is based on certain hyperbolic toral automorphisms.
\end{abstract}

\maketitle


A \emph{Liouville domain} $(W^{2n},\omega,X)$ is a triple where $W$ is a compact manifold with (nonempty) boundary, $\omega$ is a symplectic form and $X$ is a vector field, called the \emph{Liouville vector field}, such that $\mathcal{L}_X \omega=\omega$ and $X$ is outward pointing along $\p W$. Define the \emph{Liouville form} $\lambda \coloneqq i_X \omega$. Then by Cartan's formula, the symplectic form $\omega=d\lambda$ is exact. A Liouville domain $(W,\omega,X)$ is called \emph{Weinstein} if, in addition, $X$ is gradient-like with respect to a Morse function on $W$. 

It turns out that the Weinstein condition is rather restrictive on the topology of $W$. Indeed, any Weinstein domain admits a handle decomposition which contains only handles of index at most $n$. On the other hand, since the first example of McDuff \cite{McD91}, it has been known that general Liouville manifolds are not subject to such topological constraints. See \cite{Gei95,Mit95,Gei94,MNW13} for more constructions of Liouville, but not Weinstein, domains. Unfortunately, no good methods are currently available to distinguish between Liouville and Weinstein structures besides the obvious topological distinction. 

The goal of this note is produce a few (exotic) examples of Liouville domains where the dynamics of the Liouville vector field $X$ can be explicitly described. It turns out that the dynamics of $X$ is only interesting when restricted to the skeleton of $(W,\omega,X)$ which we now introduce.

Given a Liouville domain $(W,\omega,X)$, the \emph{skeleton} $\Sk(W,\omega,X)$ is defined by
\begin{equation*}
\Sk(W,\omega,X) \coloneqq \bigcap_{t>0} \phi^X_{-t}(W),
\end{equation*}
where $\phi^X_t$ denotes the time-$t$ flow of $X$. Clearly $\Sk(W,\omega,X)$ contains all the information\footnote{It is not enough, however, to merely know, say, the topological type of $\Sk(W,\omega,X)$ whose embedding in $W$ can be very wild.} of the symplectic structure, but it is, in general, not invariant under Liouville homotopies. It is a very interesting problem to understand how $\Sk(W,\omega,X)$ and $\Sk(W,\omega,X')$ are related to each other if $X,X'$ are two different Liouville vector fields on the same symplectic manifold. In the case of Weinstein domains, an on-going project of Alvarez-Gavela, Eliashberg and Nadler \cite{AGEN19} aims at simplifying $\Sk(W,\omega,X)$, up to Weinstein homotopy, such that it contains only arboreal singularities introduced by Nadler \cite{Nad17}. In the following we sometimes simply write $\Sk(W)$ for the skeleton if there is no risk of confusion.

Now let's present the main construction of this note: \emph{Liouville domains as partial mapping tori}. Let $M^{2n-1}$ be a compact manifold with boundary and $\alpha$ be a contact form on $M$. \emph{For the rest of this note, every contact manifold comes with a chosen contact form.}

\begin{definition} \label{defn:contraction}
A compact contact manifold $(M,\alpha)$ admits a \emph{contraction} if there exists a map $\phi:M \to M$, which satisfies the following properties:
\begin{itemize}
	\item[(D1)] $\phi(M) \subset \op{int} (M)$;
	\item[(D2)] $\phi$ is a diffeomorphism onto its image;
	\item[(D3)] $\phi^{\ast} (\alpha) = e^{-g} \alpha$, where $g: M \to \Rbb_{>0}$ is a positive function.
\end{itemize}
\end{definition}

Note that (D3) implies $\p M \neq \varnothing$ for volume considerations. We start with a not so interesting example.

\begin{exam} \label{ex:weinstein}
Let $Y$ be a closed manifold and consider the 1-jet space $J^1 Y$ equipped with the standard contact form $\alpha=dz-pdq$. Let $M \subset J^1 Y$ be a closed tubular neighborhood of the 0-section $Y \subset J^1 Y$. Then the map $\phi: M \to M$ defined by $\phi(z,q,p)=(z/2,q,p/2)$ is clearly a contraction. In this example, $\phi(M)$ is a deformation retract of $M$, but this is not necessarily the case in general.
\end{exam}

We construct a Liouville domain $W_{(M,\phi)}$, as a partial mapping torus, in three steps as follows. Firstly, let $\Rbb \times M$ be the symplectization of $(M,\alpha)$ with Liouville form $\lambda=e^s \alpha$, where $s \in \Rbb$; secondly, let $G:M \to \Rbb_{>0}$ be a smooth extension of the function $g \circ \phi^{-1}: \phi(M) \to \Rbb_{>0}$; finally, define the partial mapping torus
\begin{equation*}
	W_{(M,\phi)} \coloneqq \{ (s,x) \in \Rbb \times M ~|~ 0 \leq s \leq G(x) \} \big/ (0,x) \sim (G(x), \phi(x)).
\end{equation*}
We claim that $\lambda$ descends to a 1-form on $W_{(M,\phi)}$. Indeed, define $\Phi: \Rbb \times M \to \Rbb \times M$ by $\Phi(s,x) \coloneqq (s+G(x),\phi(x))$, then
\begin{equation*}
	\Phi^{\ast} (\lambda) = e^{s+G \circ \phi} \phi^{\ast} \alpha = e^{s+g} e^{-g} \alpha = \lambda
\end{equation*}
as desired. Abusing notations, we also write $\lambda$ for the descendent 1-form on $W_{(M,\phi)}$. 



Define the \emph{vertical boundary} of $W_{(M,\phi)}$ by
\begin{equation*}
\p_v W_{(M,\phi)} \coloneqq \{ (s,x) \in \Rbb \times \p M ~|~ 0 \leq s \leq G(x) \},
\end{equation*}
and the \emph{horizontal boundary} $\p_h W_{(M,\phi)}$ as the closure of $\p W_{M,\phi} \setminus \p_v W_{M,\phi}$. It follows from the construction that the Liouville vector field $X$ is outward-pointing along $\p_h W_{(M,\phi)}$ and is tangent to $\p_v W_{(M,\phi)}$. By slightly tilting $\p_v W_{(M,\phi)}$, we can assume that $X$ is everywhere outward-pointing along $\p W_{(M,\phi)}$. More precisely, let $\p M \times [-\epsilon,0] \subset M \setminus \phi(M)$ be a collar neighborhood of the boundary such that $\p M$ is identified with $\p M \times \{0\}$, where $\epsilon>0$ is small. Suppose, without loss of generality, that $G$ is constant on $\p M \times [-\epsilon,0]$. Define $P \subset [0,G] \times (\p M \times [-\epsilon,0])$ by
\begin{equation*}
P \coloneqq \{ (s,x) \in [0,G] \times (\p M \times [-\epsilon,0]) ~|~ -\epsilon s/G < \tau(x) \leq 0 \},
\end{equation*}
where $\tau: \p M \times [-\epsilon,0] \to [-\epsilon,0]$ denotes the projection. Then by construction $W_{(M,\phi)} \setminus P$ has a piecewise smooth boundary which is everywhere transverse to the Liouville vector field $X=\p_s$. Finally, one can easily round the corners on the boundary of $W_{(M,\phi)} \setminus P$ as shown in \autoref{fig:corner} to result in a smooth Liouville domain. By abuse of notation, we denote the resulting Liouville domain, again, by $W_{(M,\phi)}$. Of course as a symplectic manifold $W_{(M,\phi)}$ depends on various parameters involved in the construction above, but $W_{(M,\phi)}$ is a well-defined Liouville domain up to deformation, i.e., Liouville homotopy. In what follows we will not distinguish between Liouville domains which are deformation equivalent.

\begin{figure}[ht]
\begin{overpic}[scale=.5]{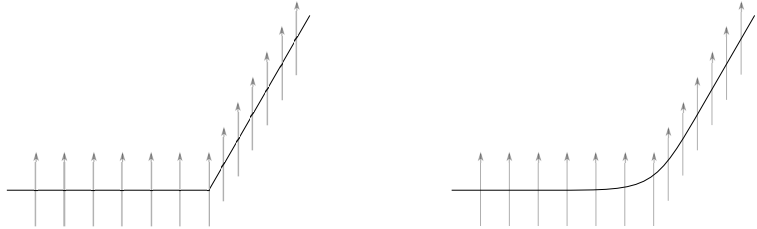}
\end{overpic}
\caption{Round the corners on $W_{(M,\phi)} \setminus P$.}
\label{fig:corner}
\end{figure}

We conclude our general construction with an obvious lemma which describes the skeleton of $W_{(M,\phi)}$.

\begin{lemma} \label{lem:skeleton of mapping torus}
The skeleton of $W_{(M,\phi)}$ is given by the mapping torus 
\begin{equation*}
\Sk(W_{(M,\phi)}) = \{ (s,x) \in \Rbb \times K ~|~ 0 \leq s \leq G(x) \} / (0,x) \sim (G(x),\phi(x)),
\end{equation*}
where $K \coloneqq \bigcap_{i \geq 0} \phi^i(M) \subset M$.
\end{lemma}

It is clear that if the input $(M,\phi)$ is as in \autoref{ex:weinstein}, then the resulting $W_{(M,\phi)}$ is Weinstein and the skeleton is a smooth Lagrangian submanifold $S^1 \times Y$. In the following, we give two explicit examples of $(M,\phi)$ of very different nature, such that the resulting Liouville domains $W_{(M,\phi)}$ are ``more interesting''.

\begin{exam}[Smale's attractor/Solenoid {\cite[Section 17.1]{KH95}}] \label{ex:smale}
Consider $M=S^1 \times D^2$ equipped with the contact form $\alpha=dx+yd\theta$, where $\theta \in S^1$ and $(x,y) \in D^2 \subset \Rbb^2$. Define $\phi:M \to M$ by
\begin{equation*}
\phi(\theta,x,y) = \left( 2\theta, \tfrac{1}{10}x+\tfrac{1}{2}\cos\theta, \tfrac{1}{2} \left( \tfrac{1}{10}y+\tfrac{1}{2}\sin\theta \right) \right).
\end{equation*}
Clearly $\phi$ satisfies (D1)--(D2). But $\phi(M)$ is not a deformation retract of $M$. Instead, it winds around $M$ twice along the $S^1$-factor. It is straightforward to compute that $\phi^{\ast} \alpha=\tfrac{1}{10}\alpha$, and therefore $\phi$ satisfies (D3).  Hence $\phi$ is a contraction and we have a well-defined Liouville domain $W_{\op{SA}}\coloneqq W_{(M,\phi)}$. 

Let's examine the skeleton $\Sk(W_{\op{SA}})$. First observe that $K = \bigcap_{i \geq 0} \phi^i(M)$ is itself a mapping torus of a Cantor set. Namely, for each fixed $\theta_0 \in S^1$, the intersection $K \cap (\{\theta_0\} \times D^2) \subset D^2$ is a Cantor set. Such $K$ is known as a \emph{solenoid} in Dynamical Systems, and is a \emph{hyperbolic attractor}. In particular, it is stable under $C^{\infty}$-small perturbations. We refer the interested readers to the comprehensive monograph \cite{KH95} for more details. In light of \autoref{lem:skeleton of mapping torus}, $\Sk(W_{\op{SA}})$ is a mapping torus of the solenoid $K$.

It follows from the construction that, as a smooth $4$-manifold, $W_{\op{SA}}$ can be built by handles of index at most $2$. On the other hand, the Hausdorff dimension of $\Sk(W_{\op{SA}})$ is strictly greater than $2$. Hence it is a very intriguing question to ask whether $W_{\op{SA}}$ is Liouville homotopic to a Weinstein domain. More generally, one can ask whether $W_{\op{SA}} \times T^{\ast} Y$ is Liouville homotopic to Weinstein for any smooth manifold $Y$.
\end{exam}

\begin{exam}[Updates on \autoref{ex:smale}]
The updates presented here were established soon after the first appearance of this note on arXiv in October 2019.

First, let's generalize \autoref{ex:smale} as follows. Let $Y$ be a compact manifold without boundary and $M \subset J^1 Y$ be a closed tubular neighborhood of the $0$-section. Then $M$ inherits the standard contact form $\alpha=dz-pdq$ on $J^1 Y$. Denote $\xi=\ker\alpha$. 

Suppose $\Lambda \subset (M,\xi)$ is a Legendrian submanifold which is diffeomorphic to $Y$. Weinstein's neighborhood theorem implies that a sufficiently small closed tubular neighborhood $N_{\epsilon}(\Lambda)$ of $\Lambda$ is contactomorphic to $(M,\xi)$. Indeed, it is easy to see that if $\epsilon$ is small, then there exists
\begin{equation*}
\phi: M \xrightarrow{\sim} N_{\epsilon}(\Lambda) \subset \op{int}(M)
\end{equation*}
such that $\phi^{\ast} \alpha = e^{-g}\alpha$ for some $g:M \to \Rbb_{>0}$. In other words, such $\phi$ is a contraction in the sense of \autoref{defn:contraction}. Hence we have the associated Liouville domain $W_{(Y,\Lambda)} \coloneqq W_{(M,\phi)}$. Clearly \autoref{ex:smale} is a special case of this construction where both $Y$ and $\Lambda$ are circles.

Next, we make the following observation: \emph{all such $W_{(Y,\Lambda)}$ are Weinstein up to homotopy}. This answers my own question above in a somewhat disappointing way. Indeed, according to the \emph{Creation Lemma} discussed in \cite[\S 9.3]{HH19}, we just need one \emph{box-fold}\footnote{Strictly speaking, one needs to slightly tilt the vertical sides of the box-fold since Liouville homotopies correspond to graphical perturbations in the contactization.} based on, say, $\phi(M)$ to make $W_{(Y,\Lambda)}$ Weinstein. For example, the Liouville domain $W_{\op{SA}}$ constructed in \autoref{ex:smale} is homotopic to a Weinstein domain which can be built by one $0$-handle, two $1$-handles and one $2$-handle. In fact, it seems most likely that there exists no nontrivial upper bound on the dimension of $\Sk(W)$, i.e., other than $\dim \Sk(W)<\dim W$, in general under the assumption that $W$ is homotopic to a Weinstein domain.

However, since no details of the argument are given here, we formulate our observation as a conjecture as follows.

\begin{conj}
The Liouville structure on $W_{(Y,\Lambda)}$, in particular $W_{\op{SA}}$ from \autoref{ex:smale}, is homotopic to a Weinstein structure.
\end{conj}

Finally, we present a variation of \autoref{ex:smale}. Namely, instead of considering the neighborhood of a Legendrian knot as in \autoref{ex:smale}, let $M=S^1 \times D^2$ be a neighborhood of a transverse knot, equipped with a contact form $\alpha = d\theta-ydx$, where $\theta \in S^1$ and $(x,y) \in D^2$. Up to rescaling, we can identify $S^1=\Rbb/\Zbb$ so that the length of the transverse knot $S^1 \times \{0\}$ is one.

Let $\Lambda \subset \op{int}(M)$ be a Legendrian knot. Then there exist coordinates $\bar{\theta},\bar{x},\bar{y}$ in an $\epsilon$-neighborhood $N_{\epsilon}(\Lambda) \cong \Rbb/\Zbb \times D^2(\epsilon)$ of $\Lambda \cong \Rbb/\Zbb \times \{0\}$ such that $\alpha|_{N_{\epsilon}(\Lambda)} = d\bar{x}+\bar{y}d\bar{\theta}$. Fix $0 < \delta \ll c \ll \epsilon$ to be specified later. Define $K \coloneqq \Rbb/\Zbb \times \{(0,\delta)\}$ to be a transverse push-off of $\Lambda$ in $N_{\epsilon}(\Lambda)$. Consider the following diffeomorphism
\begin{align*}
\psi: \Rbb/\Zbb \times D^2 &\to \Rbb/\Zbb \times D^2(\epsilon) \\
(\theta,x,y) &\mapsto \left( \theta-\tfrac{cx}{\delta}, cx, \tfrac{c\delta}{c-\delta y} \right)
\end{align*}
onto its image, which sends $S^1 \times \{0\}$ to $K$. We compute
\begin{equation*}
\psi^{\ast} \alpha = cdx + \tfrac{c\delta}{c-\delta y} (d\theta-\tfrac{c}{\delta} dx) = \tfrac{c\delta}{c-\delta y} \alpha.
\end{equation*}
It follows that $\psi$ is a contraction in the sense of \autoref{defn:contraction} if $c \gg \delta$.

The resulting Liouville domain $W_{(M,\psi)}$ is also homotopic to a Weinstein domain since $K$ is sufficiently close to a Legendrian. This can be considered as a reminiscent of a general principle that transverse knot theory is a stabilized version of Legendrian knot theory. The details of the argument, again, will be omitted.
\end{exam}

\begin{exam}[Anosov map]
Let $T^n = \Rbb^n/\Zbb^n$ be the $n$-dimensional torus and $M=D^{n-1} \times T^n$ where $D^{n-1} \subset \Rbb^{n-1}$ denotes the unit ball. Suppose $A \in SL(n,\Zbb)$ has real eigenvalues $\lambda_1,\cdots,\lambda_n$ such that $0 < \lambda_n < \abs{\lambda_i}$ for all $1 \leq i \leq n-1$. In particular $0<\lambda_n<1$. The existence of such $A$ will be established in the Appendix. View $A$ as an automorphism of $T^n$. Then there exists linear 1-forms $\beta_i, 1 \leq i \leq n$, on $T^n$ such that $A^{\ast}(\beta_i) = \lambda_i \beta_i$. In particular $\beta_i, 1 \leq i \leq n$, are linearly independent. Define the contact form
\begin{equation*}
\alpha \coloneqq \beta_n + \sum_{1 \leq i \leq n-1} y_i \beta_i
\end{equation*}
on $M$, where $(y_1,\cdots,y_{n-1}) \in D^{n-1}$. The map $\phi_A: M \to M$ defined by
\begin{equation*}
\phi_A (y_1,\cdots,y_{n-1},x) = \left( \tfrac{\lambda_n}{\lambda_1} y_1, \cdots, \tfrac{\lambda_n}{\lambda_{n-1}} y_{n-1}, Ax \right)
\end{equation*}
is clearly a contraction. Indeed, we have $\phi_A^{\ast} (\alpha) = \lambda_n \alpha$. We denote the resulting Liouville domain by $W_A$. Then the skeleton $\Sk(W_A)$ is a smooth $(n+1)$-manifold given by the mapping torus of $A: T^n \to T^n$. For $n=2$, we recover the examples of Mitsumatsu \cite{Mit95}.

An answer to the following question is desirable, but unfortunately is unknown to the author.

\begin{question}
Is $W_A$ stably homotopic to a Weinstein structure, i.e., is $W_A \times \Rbb^{2m}$ Weinstein for $m \gg 0$ ?
\end{question}

It was pointed out to the author by G. Dimitroglou-Rizell that $W_A \times \Rbb^2$, being diffeomorphic to the cotangent bundle of $\Sk(W_A)$, cannot be symplectomorphic to the cotangent bundle since everything in $W_A \times \Rbb^2$ is displaceable due to the $\Rbb^2$-factor.
\end{exam}

\begin{remark}
In a recent preprint of Eliashberg, Ogawa and Yoshiyasu \cite{EOY20}, the authors proved that every sufficiently stabilized Liouville manifold is symplectomorphic to a (flexible) Weinstein manifold. However, the symplectomorphism constructed therein is in general not compactly supported. In particular, the contact boundary at infinity is not preserved under such construction. It is still unknown, to the best of my knowledge, that whether a sufficiently stabilized Liouville manifold is necessarily compactly supported deformation equivalent to a Weinstein manifold.
\end{remark}

\appendix

\section{Toral automorphism with real spectrum}
The goal of this appendix is to construct hyperbolic toral automorphisms with real spectrum. The material presented here came from a joyful discussion with Luis Diogo, who deserves every bit of this wonderful (but maybe trivial) result, in the summer of 2019 in Uppsala.

\begin{prop} \label{prop:matrix}
Fix $n \geq 2$. For any $\epsilon>0$ and any tuple $(\mu_1,\dots,\mu_{n-2}) \in \Rbb^{n-2}$, there exists a matrix $A \in SL(n,\Zbb)$ which is diagonalizable in $SL(n,\Rbb)$ such that the eigenvalues $\lambda_i, 1 \leq i \leq n$, satisfy
\begin{enumerate}
	\item $\abs{\lambda_i-\mu_i}<\epsilon$ for $1 \leq i \leq n-2$.
	\item $\abs{\lambda_{n-1}}>1/\epsilon$ and $\abs{\lambda_n}<\epsilon$.
\end{enumerate}
\end{prop}

\begin{proof}
Using the Frobenius companion matrix, it suffices to find infinitely many tuples $\kbf \coloneqq (k_1,\dots,k_{n-1}) \in \Zbb^{n-1}$ such that the polynomial
\begin{equation*}
P_{\kbf} \coloneqq x^n - k_{n-1}x^{n-1}+\dots+(-1)^{n-1} k_1 x+(-1)^n
\end{equation*}
has $n$ roots $\lambda_1\dots\lambda_n$ which satisfy the conditions in the Proposition. Indeed $P_{\kbf}$ is the characteristic polynomial of the matrix
\begin{equation*}
	A_{\kbf} \coloneqq
	\begin{pmatrix}
		0 & 0 & \cdots & 0 & (-1)^{n+1} \\
		1 & 0 & \cdots & 0 & (-1)^n k_1 \\
		0 & 1 & \cdots & 0 & (-1)^{n-1} k_2 \\
		\vdots & \vdots & \ddots & \vdots & \vdots \\
		0 & 0 & \cdots & 1 & k_{n-1}
	\end{pmatrix}
	\in SL(n,\Zbb),
\end{equation*}
which is exactly what we look for. Let $\sigma^{(n)}_i = \sigma^{(n)}_i(\lambda_1,\cdots,\lambda_n), 1 \leq i \leq n$, be the $i$-th elementary symmetric polynomial in $n$ variables, i.e.,
\begin{equation*}
\sigma^{(n)}_i \coloneqq \sum_{1 \leq j_1 < \cdots < j_i \leq n} \lambda_{j_1} \cdots \lambda_{j_i}.
\end{equation*}
Then it suffices to find $\kbf \in \Zbb^{n-1}$ such that the system of equations 
\begin{equation} \label{eqn:lambda and k}
\sigma^{(n)}_i (\lambda_1,\cdots,\lambda_n) = k_i,\quad 1 \leq i \leq n,
\end{equation}
is satisfied, where $k_n \coloneqq 1$. Roughly speaking, the strategy consists of three steps: first, eliminate $\lambda_{n-1},\lambda_n$ from \autoref{eqn:lambda and k}; second, prescribe a ``generic'' sequence of numbers $\lambda_i$ with $\abs{\lambda_i}<1$, $1 \leq i \leq n-2$, and find $\kbf$ using ergodic theory such that \autoref{eqn:lambda and k} (without $\lambda_{n-1},\lambda_n$) is approximately satisfied; third, argue that an exact solution to \autoref{eqn:lambda and k} exists by a suitable choice of $\abs{\lambda_{n-1}}<1, \abs{\lambda_n}>1$, and a small perturbation of $\lambda_i, 1 \leq i \leq n-2$.
	
\s\n
\textsc{Step 1.} \emph{Elimination of $\lambda_{n-1},\lambda_n$ from \autoref{eqn:lambda and k}.}
	
\s
Let us rewrite \autoref{eqn:lambda and k} as follows:
\begin{align}
\sigma^{(n-2)}_1 + \lambda_{n-1} + \lambda_n &= k_1; \label{eqn:recursion 1}\\
\sigma^{(n-2)}_j + (\lambda_{n-1}+\lambda_n) \sigma^{(n-2)}_{j-1} + \lambda_{n-1} \lambda_n \sigma^{(n-2)}_{j-2} &= k_j , \quad 2 \leq j \leq n-1; \label{eqn:recursion 2} \\
\sigma^{(n-2)}_{n-2} \lambda_{n-1} \lambda_n &= 1. \label{eqn:recursion 3}
\end{align}
Here $\sigma^{(n-2)}_{n-1} \equiv 0$ and $\sigma^{(n-2)}_0 \equiv 1$ by convention. Plugging \autoref{eqn:recursion 1} and \autoref{eqn:recursion 3} into \autoref{eqn:recursion 2}, we have
\begin{equation} \label{eqn:dynamics}
\sigma^{(n-2)}_j + (k_1-\sigma^{(n-2)}_1) \sigma^{(n-2)}_{j-1} + \sigma^{(n-2)}_{j-2} \big/ \sigma^{(n-2)}_{(n-2)} = k_j , \quad 2 \leq j \leq n-1,
\end{equation}
which is a system of equations without $\lambda_{n-1}$ and $\lambda_n$. Clearly $\lambda_{n-1},\lambda_n$ can be solved easily from \autoref{eqn:recursion 1} and \autoref{eqn:recursion 3} once we determine the values of $\lambda_i, 1 \leq i \leq n-2$.
	
\s\n
\textsc{Step 2.} \emph{Passing to a discrete dynamical system.}
	
\s
The idea is that for any fixed $\lambda_1,\cdots,\lambda_{n-2}$, we can view the left-hand side of \autoref{eqn:dynamics} as a discrete dynamical system as $k_1$ runs through the integers, while the right-hand side $(k_2,\cdots,k_{n-1}) \in \Zbb^{n-2}$ forms a lattice in $\Rbb^{n-2}$. Then the existence of approximate solutions to \autoref{eqn:dynamics} is, roughly speaking, a consequence of the \emph{ergodicity} of such dynamical system.
	
The technical heart of this argument is a theorem due to Weyl and von Neumann which we now recall. See \cite[Lect. 3]{Sin76} for an excellent exposition on this topic. Let $T^m = \Rbb^m/\Zbb^m$ be an $m$-dimensional torus. Fix a vector $\rbf=(r_1,\cdots,r_m) \in \Rbb^m$. Define the translation $\tau_{\rbf}: T^m \to T^m$ by $\tau([\xbf]) = [\xbf+\rbf]$.
	
\begin{theorem}[Weyl-von Neumann] \label{thm:ergodicity}
The translation $\tau_{\rbf}$ is ergodic if and only if $\rbf$ is \emph{irrational}, i.e., the components of $\rbf$ are linearly independent over $\Zbb$.
\end{theorem}
	
In order to apply \autoref{thm:ergodicity} to our case, choose $\lambda_i, 1 \leq i \leq n-2$, such that the vector
\begin{equation*}
\rbf \coloneqq (\sigma^{(n-2)}_1, \cdots, \sigma^{(n-2)}_{n-2}) \in \Rbb^{n-2}
\end{equation*}
is irrational. For example, it suffices to choose $\lambda_i, 1 \leq i \leq n-2$, to be algebraically independent. In particular $\lambda_i \neq \lambda_j$ whenever $i \neq j$. Let's rewrite \autoref{eqn:dynamics} as 
\begin{equation} \label{eqn:simplified dynamics}
\xbf(\rbf)+k_1 \rbf = \kbf',
\end{equation} 
where $\kbf' \coloneqq (k_2,\cdots,k_{n-1}) \in \Zbb^{n-2}$ and $\xbf(\rbf) \coloneqq (x_2,\cdots,x_{n-1}) \in \Rbb^{n-1}$ with
\begin{equation*}
x_j \coloneqq \sigma^{(n-2)}_j - \sigma^{(n-2)}_1 \sigma^{(n-2)}_{j-1} + \sigma^{(n-2)}_{j-2} \big/ \sigma^{(n-2)}_{(n-2)},
\end{equation*}
for $2 \leq j \leq n-1$. It follows from \autoref{thm:ergodicity} that for any $\epsilon>0$ and $K>0$, there exists $k_1>K$ and $\kbf' \in \Zbb^{n-2}$ such that $\abs{\xbf + k_1 \rbf - \kbf'}<\epsilon$. Switching point of view, one can think of the prescribed tuple $(\lambda_1, \cdots, \lambda_{n-2})$ as an approximate solution to \autoref{eqn:dynamics} with suitable choices of $k_i, 1 \leq i \leq n-1$.
	
\s\n
\textsc{Step 3.} \emph{From approximate solutions to exact solutions.}
	
\s
Observe that the tuple $(\lambda_1,\cdots,\lambda_{n-2})$ uniquely determines the vector $\rbf$ which approximately solves \autoref{eqn:simplified dynamics} with $k_1>K$. By choosing $K$ sufficiently large, there exists an exact solution $\rbf'$ of \autoref{eqn:simplified dynamics} which is close to $\rbf$. It remains to argue that $\rbf'$ corresponds to a tuple $(\lambda'_1,\cdots,\lambda'_{n-2})$ which is close to $(\lambda_1,\cdots,\lambda_{n-2})$. Indeed, consider the map $\Pi: \Rbb^{n-2} \to \Rbb^{n-2}$ defined by
\begin{equation*}
\Pi(\lambda_1,\cdots,\lambda_{n-2}) = (\sigma^{(n-2)}_1, \cdots, \sigma^{(n-2)}_{n-2}).
\end{equation*}
The Jacobian Jac$(\Pi) \neq 0$ if $\lambda_i, 1 \leq i \leq n-2$, are algebraically independent. By the Inverse Function Theorem, $\Pi^{-1} (\rbf')$ exists and is close to $(\lambda_1,\cdots,\lambda_{n-2})$. Abusing notations, let us write $\Pi^{-1}(\rbf') = (\lambda_1,\cdots,\lambda_{n-2})$.
	
To wrap up the proof, let us rewrite \autoref{eqn:recursion 1} and \autoref{eqn:recursion 3} as follows:
\begin{equation} \label{eqn:last two lambdas}
\lambda_{n-1}+\lambda_n = k_1 - \sigma^{(n-2)}_1, \quad \lambda_{n-1} \lambda_n = 1 \big/ \sigma^{(n-2)}_{n-2}.
\end{equation}
Since both $\sigma^{(n-2)}_1$ and $\sigma^{(n-2)}_{n-2} \neq 0$ are finite numbers, \autoref{eqn:last two lambdas} admits a solution $(\lambda_{n-1},\lambda_n)$ with $\abs{\lambda_{n-1}}>1/\epsilon$ and $\abs{\lambda_n}<\epsilon$ as long as $k_1>K$ is sufficiently large.
\end{proof}

\s\n
{\em Acknowledgements.}  The author is grateful to Ko Honda for many years of collaboration which essentially shaped his understanding of contact and symplectic structures. He also wants to thank his friends and colleagues at Uppsala and Nantes for their interest and curiosity in this work. Finally, the author thanks an anonymous referee for his/her comments on the first draft.

\bibliography{mybib}

\providecommand{\bysame}{\leavevmode\hbox to3em{\hrulefill}\thinspace}
\providecommand{\MR}{\relax\ifhmode\unskip\space\fi MR }
\providecommand{\MRhref}[2]{%
  \href{http://www.ams.org/mathscinet-getitem?mr=#1}{#2}
}
\providecommand{\href}[2]{#2}
\begin{thebibliography}{MNW13}

\bibitem[AGEN]{AGEN19}
Daniel Alvarez-Gavela, Yakov Eliashberg, and David Nadler, \emph{Geomorphology
  of {L}agrangian ridges}, preprint 2019, arXiv:1912.03439.

\bibitem[EOY]{EOY20}
Yakov Eliashberg, Noboru Ogawa, and Toru Yoshiyasu, \emph{Stabilized convex
  symplectic manifolds are {W}einstein}, preprint 2020, arXiv:2003.12251.

\bibitem[Gei94]{Gei94}
Hansj\"{o}rg Geiges, \emph{Symplectic manifolds with disconnected boundary of
  contact type}, Internat. Math. Res. Notices (1994), no.~1, 23--30.
  \MR{1255250}

\bibitem[Gei95]{Gei95}
\bysame, \emph{Examples of symplectic {$4$}-manifolds with disconnected
  boundary of contact type}, Bull. London Math. Soc. \textbf{27} (1995), no.~3,
  278--280. \MR{1328705}

\bibitem[HH]{HH19}
Ko~Honda and Yang Huang, \emph{Convex hypersurface theory in contact topology},
  preprint 2019, arXiv:1907.06025.

\bibitem[KH95]{KH95}
Anatole Katok and Boris Hasselblatt, \emph{Introduction to the modern theory of
  dynamical systems}, Encyclopedia of Mathematics and its Applications,
  vol.~54, Cambridge University Press, Cambridge, 1995, With a supplementary
  chapter by Katok and Leonardo Mendoza. \MR{1326374}

\bibitem[McD91]{McD91}
Dusa McDuff, \emph{Symplectic manifolds with contact type boundaries}, Invent.
  Math. \textbf{103} (1991), no.~3, 651--671. \MR{1091622}

\bibitem[Mit95]{Mit95}
Yoshihiko Mitsumatsu, \emph{Anosov flows and non-{S}tein symplectic manifolds},
  Ann. Inst. Fourier (Grenoble) \textbf{45} (1995), no.~5, 1407--1421.
  \MR{1370752}

\bibitem[MNW13]{MNW13}
Patrick Massot, Klaus Niederkr\"{u}ger, and Chris Wendl, \emph{Weak and strong
  fillability of higher dimensional contact manifolds}, Invent. Math.
  \textbf{192} (2013), no.~2, 287--373. \MR{3044125}

\bibitem[Nad17]{Nad17}
David Nadler, \emph{Arboreal singularities}, Geom. Topol. \textbf{21} (2017),
  no.~2, 1231--1274. \MR{3626601}

\bibitem[Sin76]{Sin76}
Yakov~G. Sinai, \emph{Introduction to ergodic theory}, Princeton University
  Press, Princeton, N.J., 1976, Translated by V. Scheffer, Mathematical Notes,
  18. \MR{0584788}

\end{thebibliography}
\bibliographystyle{amsalpha}

\end{document}